\documentclass[12pt]{amsart}
\usepackage{amsmath,amsthm,amsfonts,amscd,amssymb,eucal,latexsym, fullpage}
\usepackage{mathrsfs}
\usepackage[numbers]{natbib}
\usepackage[fit]{truncate}
\usepackage{fullpage}
\usepackage{graphicx}
\usepackage{epsfig}
\usepackage{color,soul}
\usepackage[all,cmtip]{xy}
\usepackage[top=1.125in, bottom=1.125in, left=1.25in, right=1.25in]{geometry}
\newcommand{\truncateit}[1]{\truncate{0.8\textwidth}{#1}}
\newcommand{\scititle}[1]{\title[\truncateit{#1}]{#1}}

\theoremstyle{plain}
\newtheorem{theorem}{Theorem}[section]
\newtheorem{corollary}[theorem]{Corollary}
\newtheorem{lemma}[theorem]{Lemma}

\theoremstyle{definition}

\newtheorem{remark}[theorem]{Remark}

\def\ti{\tilde}       
   \def\ep{\varepsilon}

\newcommand{\R}{\mathbb{R}}

\newcommand{\Z}{\mathbb{Z}}

\newcommand{\Mm}{\mathcal{M}}
\newcommand{\Bb}{\mathcal{B}}
\newcommand{\Nn}{\mathcal{N}}
\newcommand{\Zz}{\mathcal{Z}}
\newcommand{\Cc}{\mathcal{C}}

\newcommand{\Tt}{\mathcal{T}}
\newcommand{\Oo}{\mathcal{O}}

\newcommand{\Ee}{\mathcal{E}}

\newcommand{\Nc}{\mathfrak{N}}

\makeatletter
\newcommand*\owedge{\mathpalette\@owedge\relax}
\newcommand*\@owedge[1]{%
  \mathbin{%
    \ooalign{%
      $#1\m@th\bigcirc$\cr
      \hidewidth$#1\m@th\wedge$\hidewidth\cr
    }%
  }%
}
\makeatother

\DeclareMathOperator{\mass}{mass}

\DeclareMathOperator{\diam}{diam}

\DeclareMathOperator{\length}{length}

\DeclareMathOperator{\HF}{FH}

\DeclareMathOperator{\vol}{vol}
\DeclareMathOperator{\Vol}{Vol}
\DeclareMathOperator{\Ein}{Ein}

\DeclareMathOperator{\Ric}{Ric}
\DeclareMathOperator{\Rm}{Rm}

\scititle{Homological Filling and Minimal Varifolds\\
in Four-Dimensional Einstein Manifolds}

\author{Wenjie Fu and Zhifei Zhu}
\date{}

\begin{document}
\maketitle

\begin{abstract}
We study the smallest area $A_{\min}(M,g)$ of a $2$-dimensional stationary integral varifold in a closed Einstein $4$-manifold $(M^4,g)$ with
\[
\Ric_g = \lambda g,\quad |\lambda|\le 3,\quad \Vol(M,g)\ge v>0,\quad \diam(M,g)\le D,\quad H_1(M;\Z)=0.
\]
Building on the previous work \cite{LW_HF1} on homological filling functions, we show that for every $(M^4,g)$ in this Einstein class there is an upper bound
\[
A_{\min}(M,g)\le F_{\Ein}(v,D),
\]
where $F_{\Ein}$ depends only on $(v,D)$. The proof uses the Einstein-specific analytic input of a uniform Sobolev inequality, a global $L^2$-curvature bound, and standard $\varepsilon$-regularity/harmonic-radius estimates.
\end{abstract}

\section{Introduction}

Let $(M^4,g)$ be a closed Riemannian manifold and denote by $A_{\min}(M,g)$ 
the area of a smallest $2$-dimensional stationary integral varifold in $(M,g)$.
In general, the existence of such a minimal object follows from the
Almgren--Pitts min--max theory \cite{pitts2014existence}, and in low
codimension and low dimension one can often upgrade the varifold to a smooth
embedded minimal hypersurface using the regularity theory of
Schoen--Simon--Yau \cite{schoen1975,schoen1981regularity}.

In this article we are interested in bounding $A_{\min}(M,g)$ in terms of
global geometric data in the special case when $(M,g)$ is an Einstein
$4$-manifold subject to natural non-collapsing and diameter bounds.  
Throughout we assume that
\begin{equation}\label{eq:Ein-assumptions}
 \Ric_g = \lambda g,\quad |\lambda|\le 3,\quad \Vol(M,g)\ge v>0,\quad
 \diam(M,g)\le D,\quad H_1(M;\Z)=0.
\end{equation}
We write $\Ee_1(4,v,D)$ for the class of all closed Einstein $4$-manifolds
satisfying \eqref{eq:Ein-assumptions}.

This paper is a sequel to our previous work \cite{LW_HF1}, where we considered
the larger class with $ |{\rm Ric}_g|\le 3,\ \Vol(M,g)\ge v>0,
    \diam(M,g)\le D,\ H_1(M;\Z)=0, $
and proved an affine upper bound for the first homological filling function
$\HF_1$ of $(M,g)$.  More precisely, for every $(M^4,g)\in\Mm_1(4,v,D)$ we
showed that there exist functions $f_1(v,D),f_2(v,D)>0$ such that
\[
 \HF_1(l)\ \le\ f_1(v,D)\,l + f_2(v,D)\qquad\text{for all }l\ge 0.
\]
Combined with an effective version of the Almgren--Pitts min--max theory due to
Nabutovsky--Rotman \cite{nabutovsky2006curvature}, this yields
\[
 A_{\min}(M,g)\ \le\ F(v,D)
\]
for some function $F$ depending only on $v$ and $D$.  In that setting, however,
the dependence of $f_1,f_2$ and $F$ on the geometric parameters was not made
explicit, since the constants coming from the Cheeger--Naber bubble--tree
decomposition \cite{cheeger2014regularity} were only known to exist abstractly.

In the present work we revisit this question in the more rigid Einstein setting.
For $(M^4,g)\in\Ee_1(4,v,D)$ the curvature solves an elliptic system, and
global Sobolev inequalities can be used, following Anderson
\cite{anderson1989ricci,anderson1992thel}, to derive quantitative $L^2$–curvature
bounds and $\varepsilon$–regularity with constants depending only on a Sobolev
constant.  This analytic input, combined with the Einstein specialization of the
Cheeger--Naber bubble--tree decomposition and an improved combinatorial filling
estimate, allows us to obtain a quantitative upper bound for $A_{\min}(M,g)$ in
terms of $(v,D)$ and the Sobolev constants.

Our first main result is the following quantitative bound for minimal
$2$-varifolds in Einstein $4$-manifolds.

\begin{theorem}\label{thm:main}
For every $v,D>0$ there exists a constant $F_{\Ein}(v,D)>0$ with the following
property.  If $(M^4,g)\in \Ee_1(4,v,D)$, then the area $A_{\min}(M,g)$ of a
smallest $2$-dimensional stationary integral varifold in $(M,g)$ satisfies
\[
 A_{\min}(M,g)\ \le\ F_{\Ein}(v,D).
\]
Moreover, the dependence of $F_{\Ein}(v,D)$ is reduced to the Sobolev
constant $C_S(v,D)$, the global $L^2$-curvature bound, and the standard
Einstein $\varepsilon$-regularity constants.
\end{theorem}

As in \cite{LW_HF1}, the proof of Theorem~\ref{thm:main} proceeds via a
homological filling inequality expressed in terms of the first homological
filling function $\HF_1$ of $(M,g)$.

\begin{theorem}\label{thm:Ein-fill}
For every $v,D>0$ there exist functions $$f_1^{\Ein}(v,D),f_2^{\Ein}(v,D)>0, $$
with the following property.  If $(M^4,g)\in \Ee_1(4,v,D)$ and
$C\in\Zz_1(M;\Z)$ is a singular Lipschitz $1$-cycle, then $C$ bounds a
singular $2$-chain $E\in\Cc_2(M;\Z)$ such that
\[
 \mass_2(E)\ \le\ f_1^{\Ein}(v,D)\,\mass_1(C) + f_2^{\Ein}(v,D).
\]
Equivalently, the first homological filling function of $(M,g)$ satisfies
\[
 \HF_1(l)\ \le\ f_1^{\Ein}(v,D)\,l + f_2^{\Ein}(v,D)
 \qquad\text{for all }l\ge 0.
\]
\end{theorem}

Theorem~\ref{thm:main} is then an immediate consequence of
Theorem~\ref{thm:Ein-fill} and the min–max theorem of Nabutovsky--Rotman
\cite{nabutovsky2006curvature}, which relates $A_{\min}(M,g)$ to $\HF_1$.
In the Einstein setting we can track the dependence of the constants arising from the Sobolev
inequality, the $L^2$-curvature bound and the bubble-tree decomposition,
thus reducing the dependence of $F_{\Ein}(v,D)$ to these analytic and combinatorial inputs.

\section{Sobolev inequality and Einstein $\varepsilon$-regularity}\label{sec:Einstein_prelim}

Throughout this section we assume $(M^4,g)\in \Ee_1(4,v,D)$.

\subsection{Global Sobolev inequality}

We use the $L^2$-Sobolev inequality on $(M,g)$, which is equivalent to a uniform
isoperimetric inequality under our assumptions.

\begin{theorem}[Sobolev inequality]\label{thm:sobolev}
There exists a constant $C_S(v,D)>0$ depending only on $v,D$ such that for every
$u\in C^\infty(M)$ we have
\begin{equation}\label{eq:sobolev}
 \Bigl(\int_M |u|^{4}\Bigr)^{1/2}
 \ \le\ C_S(v,D)\left( \int_M |\nabla u|^2 + \frac{1}{D^2}\int_M u^2\right).
\end{equation}
\end{theorem}

\begin{proof}
Under our standing assumptions we have $\Ric_g\ge -3g$, $\Vol(M)\ge v$ and
$\diam(M)\le D$. A standard consequence is a uniform isoperimetric inequality,
hence a uniform Sobolev inequality for mean-zero functions; see, for instance,
\cite{Hebey2000nonlinear,anderson1989ricci}. Thus there exists $C'_S(v,D)$ such that
\[
 \|u-\bar u\|_{L^4(M)}^2 \le C'_S(v,D)\,\|\nabla u\|_{L^2(M)}^2,
 \qquad \bar u:=\Vol(M)^{-1}\int_M u.
\]
In addition, Bishop--Gromov comparison with $\Ric_g\ge -3g$ and $\diam(M)\le D$
gives an upper volume bound $\Vol(M)\le V_+(D)$. Since
\[
 |\bar u|^2 \le \Vol(M)^{-1}\|u\|_{L^2(M)}^2 \le v^{-1}\|u\|_{L^2(M)}^2,
\]
we obtain
\[
 \|\bar u\|_{L^4(M)}^2
 = |\bar u|^2\,\Vol(M)^{1/2}
 \le v^{-1}V_+(D)^{1/2}\,\|u\|_{L^2(M)}^2
 \le C''_S(v,D)\,\frac{1}{D^2}\|u\|_{L^2(M)}^2,
\]
after enlarging the constant. Combining the bounds for $u-\bar u$ and $\bar u$
yields \eqref{eq:sobolev}.
\end{proof}

\begin{remark}
The existence of such a Sobolev constant under a Ricci lower bound, a volume lower bound and
a diameter upper bound is classical; see, for example,\cite{Hebey2000nonlinear}.
\end{remark}

\subsection{$L^2$ curvature and the curvature equation}

We use the local $L^2$ curvature estimate of Cheeger--Naber together with a covering argument to obtain a global bound in our setting.

\begin{theorem}[$L^2$ curvature bound]\label{thm:L2}
There exists $C_{L^2}(v,D)>0$ such that
\[
 \int_M |\Rm_g|^2 \le C_{L^2}(v,D)
 \qquad \text{for all }(M^4,g)\in\Ee_1(4,v,D).
\]
\end{theorem}

\begin{proof}
Since $|\Ric_g|\le 3$, Bishop--Gromov comparison and the bounds
$\Vol(M)\ge v$, $\diam(M)\le D$ imply that there exists $v_1(v,D)>0$ such that
\[
 \Vol(B_1(x))\ge v_1(v,D)
 \qquad\text{for every }x\in M.
\]
Applying the local $L^2$ estimate of Cheeger--Naber
\cite[Theorem~1.13]{cheeger2014regularity} to each unit ball gives
\[
 \int_{B_1(x)} |\Rm|^2 \le C_{\mathrm{CN}}(v_1(v,D))
 \qquad\text{for every }x\in M.
\]
Now choose a maximal $\frac12$-separated set $\{x_1,\dots,x_N\}\subset M$.
Then the balls $B_1(x_i)$ cover $M$, while the balls $B_{1/4}(x_i)$ are disjoint.
Using again Bishop--Gromov comparison together with the diameter bound, we obtain
$N\le N(v,D)$ for some constant depending only on $(v,D)$. Therefore
\[
 \int_M |\Rm|^2
 \le \sum_{i=1}^N \int_{B_1(x_i)} |\Rm|^2
 \le N(v,D)\,C_{\mathrm{CN}}(v_1(v,D)).
\]
This gives the desired global bound.
\end{proof}

Next we recall the elliptic equation satisfied by the curvature tensor. For an Einstein
metric $\Ric_g=\lambda g$, recall the Weitzenb\"{o}ck formula for $\Rm$:
\begin{equation}\label{eq:PDE-Rm}
 \Delta \Rm = \Rm * \Rm + \lambda\,\Rm,
\end{equation}
where $\Delta$ is the rough Laplacian and $\Rm * \Rm$ denotes a quadratic expression in
$\Rm$; see for example \cite{besse1978einstein}. Let
\[
 u := |\Rm|.
\]
Then, using the standard Bochner formula and Kato's inequality $|\nabla u|\le |\nabla \Rm|$,
one obtains a scalar differential inequality of the form
\begin{equation}\label{eq:PDE-u}
 -\Delta u \ \le\ C_0\,(u^2+u)
\end{equation}
for some dimensional constant $C_0>0$. We briefly justify this.

Indeed, compute
\[
 \frac12\Delta u^2
 = \langle \Delta \Rm,\Rm\rangle + |\nabla \Rm|^2.
\]
Using \eqref{eq:PDE-Rm} and the Cauchy--Schwarz inequality,
\[
 \langle \Delta \Rm,\Rm\rangle
 = \langle \Rm*\Rm,\Rm\rangle + \lambda |\Rm|^2
 \ge -C\,|\Rm|^3 - C\,|\Rm|^2
\]
for some dimensional constant $C>0$. Therefore
\[
 \frac12\Delta u^2 \ge -C u^3 - C u^2 + |\nabla \Rm|^2.
\]
On the other hand,
\[
 \frac12\Delta u^2
 = u\Delta u + |\nabla u|^2.
\]
Combining these identities and using Kato's inequality $|\nabla u|\le|\nabla\Rm|$ yields
\[
 u\Delta u + |\nabla u|^2
 \ge -C u^3 - C u^2 + |\nabla u|^2,
\]
hence
\[
 \Delta u \ge -C(u^2+u).
\]
Rewriting this as $-\Delta u\le C_0(u^2+u)$ with $C_0=C$ gives \eqref{eq:PDE-u}.

\subsection{Einstein $\varepsilon$-regularity}

The next result is standard in dimension four; we record it in the form needed
later.

\begin{theorem}[Einstein $\varepsilon$-regularity]\label{thm:eps_reg_Einstein}
There exist constants $\varepsilon_{\mathrm{Ein}}(v,D)>0$,
$C_{\mathrm{Ein}}(v,D)>0$ and $c_{\mathrm{Ein}}(v,D)>0$ such that the
following holds.

Let $(M^4,g)\in \Ee_1(4,v,D)$ and suppose that $B_{2r}(x)\subset M$ satisfies
\[
 \int_{B_{2r}(x)} |\Rm|^2 \le \varepsilon_{\mathrm{Ein}}(v,D).
\]
Then
\[
 \sup_{B_r(x)} |\Rm| \le C_{\mathrm{Ein}}(v,D)\, r^{-2},
 \qquad
 r_h(y) \ge c_{\mathrm{Ein}}(v,D)\,r \quad\text{for all } y\in B_r(x).
\]
\end{theorem}

\begin{proof}
After rescaling we may assume $r=1$. The hypotheses place us in the standard
noncollapsed Einstein setting: the Ricci tensor is uniformly bounded, local
volume ratios are bounded below in terms of $(v,D)$, and the Sobolev constant
is controlled by Theorem~\ref{thm:sobolev}. The usual $\varepsilon$-regularity
theorem for Einstein $4$-manifolds therefore applies and yields the pointwise
curvature bound together with the harmonic-radius estimate on $B_1(x)$; see
Anderson \cite{anderson1989ricci,anderson1992thel}, Cheeger--Tian
\cite{cheegerTian2006}, and Cheeger--Naber \cite{cheeger2014regularity}. 
Rescaling back proves the statement.
\end{proof}

\begin{remark}
In what follows we only use the existence of the constants in
Theorems~\ref{thm:sobolev}, \ref{thm:L2} and \ref{thm:eps_reg_Einstein}. The
global harmonic-radius scale entering the bubble--tree decomposition will be
taken directly from Cheeger--Naber's theorem.
\end{remark}

\subsection{Bubble--tree decomposition}

We state Cheeger--Naber's structure theorem \cite{cheeger2014regularity}, which will be used to construct a cover on the manifold.

\begin{theorem}[Bubble--tree decomposition]\label{thm:bubble_tree_Ein}
There exist constants $r_0(v,D)>0$, $C_\Gamma(v,D)>0$ and integers
$N_E(v,D)$, $k_E(v,D)$ together with a decomposition
\[
 M = \Bb^1 \cup \bigcup_{j_2=1}^{N_2}\Nn^{2}_{j_2}
       \cup \bigcup_{j_2=1}^{N_2}\Bb^{2}_{j_2}
       \cup \dots
       \cup \bigcup_{j_k=1}^{N_k}\Nn^{k}_{j_k}
       \cup \bigcup_{j_k=1}^{N_k}\Bb^{k}_{j_k}
\]
such that:
\begin{enumerate}
\item[(i)] (\emph{Bodies}) If $x\in \Bb^{i}_{j}$ then
\[
 r_h(x)\ge r_0(v,D)\,\diam(\Bb^{i}_{j}).
\]
\item[(ii)] (\emph{Necks}) Each $\Nn^{i}_{j}$ is diffeomorphic to an annulus
in $\R^4/\Gamma^{i}_{j}$, with $|\Gamma^{i}_{j}|\le C_\Gamma(v,D)$.
\item[(iii)] Each $\Nn^{i}_{j}$ meets exactly two bodies, a parent
$\Bb^{i-1}_{\cdot}$ and a child $\Bb^i_j$, and their intersections are
diffeomorphic to $\R\times S^3/\Gamma^{i}_j$.
\item[(iv)] The constants $N_i\le N_E(v,D)$ and $k\le k_E(v,D)$.
\end{enumerate}
\end{theorem}

This is the specialization of Cheeger--Naber's bubble--tree construction
\cite[Theorem~8.63]{cheeger2014regularity} to the Einstein class. The displayed
decomposition above is their formula~(8.64). We denote the adjacency relations
between bodies and necks as a rooted tree with root $\Bb^1$ and denote by
$\Tt^i_j$ the subtree with root $\Bb^i_j$.

\section{Filling functions and stationary varifolds}\label{sec:HF1}

In this section we recall how an upper bound for the first homological filling
function controls the area of a smallest $2$-dimensional stationary integral
varifold, and we explain how to obtain such an upper bound in the Einstein
setting. The overall strategy follows our previous work~\cite{LW_HF1}: we
construct a cover adapted to the bubble--tree decomposition, pass to the
nerve and an associated geodesic graph, and then fill $1$-cycles by combining
local geometric fillings with a combinatorial filling argument.

The one new ingredient in the present paper is an improved
combinatorial filling lemma (Lemma~\ref{lem:comb-improved} below), based on
a sharp bound for integer solutions of linear Diophantine systems due to
Borosh--Flahive--Rubin--Treybig~\cite{BFRT1989} together with Hadamard's
inequality. For the reader's convenience we include a complete proof of this
lemma. All other steps are essentially the same as those in~\cite{LW_HF1},
up to the replacement of the general Ricci bounds by the Einstein
$\varepsilon$-regularity input from Section~\ref{sec:Einstein_prelim}, and
we therefore only sketch the arguments or omit proofs with precise references
to~\cite{LW_HF1}.

\subsection{Homological filling functions and stationary varifolds}

Let $\Cc_k(M;\Z)$ be the group of singular Lipschitz $k$-chains with integer
coefficients and $\Zz_k(M;\Z)=\ker\partial_k$ the group of $k$-cycles.
If $C=\sum_i a_i f_i\in\Cc_k(M;\Z)$ with $f_i:\Delta^k\to M$ Lipschitz,
we define its $k$-mass by
\[
 \mass_k(C)
 \;=\;\sum_i |a_i|\int_{\Delta^k} f_i^*(d\vol_g).
\]
Assuming $H_1(M;\Z)=0$, the \emph{first homological filling function} is
defined for $l\ge0$ by
\[
 \HF_1(l)
 \;=\;\sup_{\mass_1(Z)\le l}\;\inf_{\partial C=Z}\,\mass_2(C),
\]
where the supremum is taken over all singular Lipschitz $1$-cycles
$Z\in\Zz_1(M;\Z)$ with $\mass_1(Z)\le l$.

The connection between $\HF_1$ and minimal $2$-varifolds is provided by the
following theorem of Nabutovsky--Rotman.

\begin{theorem}[Nabutovsky--Rotman~{\cite{nabutovsky2006curvature}}]\label{thm:NR}
Let $M$ be a closed $n$-dimensional Riemannian manifold with $H_1(M;\Z)=0$
and diameter $\diam(M)=D$.  Let $A_{\min}(M)$ denote the area of a smallest
$2$-dimensional stationary integral varifold in $M$. Then
\[
 A_{\min}(M)\;\le\;\frac{(n+1)!}{2}\,\HF_1(2D).
\]
In particular, for $n=4$,
\[
 A_{\min}(M)\;\le\;60\,\HF_1(2D).
\]
\end{theorem}

Thus, in order to bound $A_{\min}(M)$ from above, it suffices to obtain a
affine upper bound for $\HF_1(l)$ in terms of the geometric data of $(M,g)$.
The rest of this section is devoted to establishing such a bound in the
Einstein setting.

\subsection{A homological filling inequality for Einstein manifolds}
\label{sec:filling_Einstein}

We now explain the geometric-combinatorial construction that leads to the
Einstein filling inequality. The starting point is the bubble--tree
decomposition of Theorem~\ref{thm:bubble_tree_Ein}, which decomposes
$(M^4,g)\in\Ee_1(4,v,D)$ into finitely many body regions $\Bb^i_j$ and neck
regions $\Nn^i_j$, with quantitative control on harmonic radius in the bodies
and on the topology and geometry of the necks.

The idea is to build:

\begin{itemize}
\item[(a)] a cover of $M$ by harmonic balls inside bodies and by ``trapezoids''
           inside necks, with uniformly controlled overlap,
\item[(b)] the nerve $\Nc$ of this cover and a canonical map $f:M\to\Nc$, and
\item[(c)] a geodesic graph $\Gamma\subset M$ whose vertices correspond to the
           elements of the cover and whose edges have uniformly bounded length.
\end{itemize}

A general $1$-cycle $C$ in $M$ will first be homologically replaced by a sum
of cycles supported in single bodies, then approximated by a cycle in
$\Gamma$, then straightened via the standard graph-to-nerve procedure of
\cite[Sec.~4]{LW_HF1} to a simplicial cycle in $\Nc$ and filled combinatorially
using Lemma~\ref{lem:comb-improved}.  Each $2$-simplex in $\Nc$ corresponds to
a geodesic triangle in $\Gamma$, which can be filled by a $2$-chain of
uniformly bounded area. This yields the desired affine filling inequality.

\smallskip\noindent
\emph{Bodies.}
For each body $\Bb^i_j$, let
\[
 r^i_j:=\inf_{x\in \Bb^i_j} r_h(x)\;\ge\; r_0(v,D)\,\diam(\Bb^i_j),
\]
where $r_h$ is the harmonic radius and $r_0(v,D)$ is the constant from
Theorem~\ref{thm:bubble_tree_Ein}. Set
\[
 R^i_j:=\frac{r^i_j}{64}.
\]
By volume comparison and the lower bound on $r^i_j$, one can cover $\Bb^i_j$
by at most $N_1(v,D)$ balls of radius $R^i_j$, and we denote a fixed such
subcover by $\Oo(\Bb^i_j)$.

\smallskip\noindent
\emph{Necks.}
For each neck
\(
 \Nn^i_j\simeq A_{\rho,2\rho}(0)\subset \R^4/\Gamma^i_j
\)
via a diffeomorphism $\Phi^i_j$, we cover the spherical factor
$S^3/\Gamma^i_j$ by $m_{i,j}\le N_2(v,D)$ convex balls $B_{r_c/4}(z_\ell)$
and define ``trapezoids''
\[
 T_{i,j,\ell}
 =\Phi^i_j\bigl((2\sqrt{1-\ep}\,\rho,(2-\sqrt{1-\ep})\,\rho)
               \times B_{r_c/4}(z_\ell)\bigr),
\]
contained in slightly larger convex regions
\[
 \bar T_{i,j,\ell}
 =\Phi^i_j\bigl((\rho/2,2\rho)\times B_{r_c}(z_\ell)\bigr),
\]
where $\ep>0$ and $r_c>0$ are chosen (as in~\cite{LW_HF1}) so that the
$\bar T_{i,j,\ell}$ are contained in the neck and have controlled overlap.

\smallskip
We then define the global cover
\[
 \Oo
 :=\Big(\bigcup_{i,j}\Oo(\Bb^i_j)\Big)
   \;\cup\;\Big(\bigcup_{i,j,\ell} T_{i,j,\ell}\Big).
\]
The number of sets in $\Oo$ is bounded by
\(
 |\Oo|\le \ti N_E(v,D)
\)
for some $\ti N_E(v,D)$ depending only on the bubble--tree complexity
$N_E(v,D)$, $k_E(v,D)$ and the covering constants $N_1(v,D)$, $N_2(v,D)$; see
also~\cite[Sec.~4]{LW_HF1} for the general Ricci-bounded case.

Let $\Nc$ be the nerve of $\Oo$, and let $f:M\to \Nc$ denote the canonical map
associated to a fixed partition of unity subordinate to $\Oo$.

\smallskip\noindent
\emph{Geodesic graph $\Gamma$.}
We now construct a geodesic graph $\Gamma\subset M$ whose vertices correspond
to the centers of the body balls in $\Oo(\Bb^i_j)$ and to chosen reference
points in the trapezoids $T_{i,j,\ell}$. Whenever two elements of $\Oo$
intersect, we connect the corresponding vertices in one of the following ways:
\begin{itemize}
\item[(i)] if both sets lie in the same body, we join the centers by a
           minimizing geodesic inside that body;
\item[(ii)] if both sets are trapezoids in the same neck, we join the reference
            points by a short curve contained in the union of the two
            trapezoids;
\item[(iii)] if one set is a body ball and the other is a trapezoid in a
            neighbouring neck, we concatenate a short geodesic segment inside
            the ball with a short curve inside the trapezoid.
\end{itemize}
By construction, each edge of $\Gamma$ has length bounded by
\[
 \le C(v,D)\,\max\{R^i_j,\rho\},
\]
for some constant $C(v,D)$ depending only on the geometric data.
This construction is identical to that in~\cite[Sec.~4]{LW_HF1}, with the
only difference that the scales $R^i_j$ and $\rho$ are now tied to the
harmonic-radius scale $r_0(v,D)$ furnished by Theorem~\ref{thm:bubble_tree_Ein}.

\subsection{Local fillings}

We next recall the local geometric fillings that will be used to fill the
small pieces of a $1$-cycle inside bodies and necks.

\begin{lemma}\label{lem:harm-Lip}
In a harmonic chart $\Phi:B_{r_h}(0)\to M$ with
\[
 \|g-\delta\|_{C^0}+r_h\|\partial g\|_{C^0}\le 10^{-3},
\]
any curve $\gamma$ satisfies
\[
 \sqrt{1-10^{-3}}\;\length(\gamma)
 \;\le\; \length(\Phi\gamma)
 \;\le\; \sqrt{1+2\cdot 10^{-3}}\;\length(\gamma),
\]
and $\Phi(B_{r_h}(0))\supset B_{r_h/2}(\Phi(0))$.
\end{lemma}

\begin{lemma}\label{lem:fill-ball}
Let $r_h(x)>0$ and let $C$ be a $1$-cycle whose image is contained in
$B_{r_h(x)/2}(x)$. Then there exists a $2$-chain $E$ with $\partial E=C$ and
\[
 \mass_2(E)\;\le\; r_h(x)\,\mass_1(C).
\]
\end{lemma}

\begin{proof}
In harmonic coordinates $\Phi:B_{r_h}\to M$, the metric is
$(1\pm 10^{-3})$-bilipschitz to the Euclidean metric by
Lemma~\ref{lem:harm-Lip}. Coning $\Phi^{-1}(C)$ to $0$ in $\R^4$ produces a
$2$-chain of Euclidean area $\le \frac12 r_h\,\length(\Phi^{-1}C)$.
Pushing forward via $\Phi$ distorts area by at most a fixed factor close to
$1$, which we absorb into the constant.
\end{proof}

\begin{lemma}[Neck contraction]\label{lem:neck-contraction}
Let $\Nn\simeq A_{\rho,2\rho}(0)\subset\R^4/\Gamma$ be a neck region in the
bubble--tree decomposition, and let $C$ be a $1$-cycle with image in $\Nn$.
Then there exist a $1$-cycle $C'$ supported in $\Nn\cap \Bb$ (the adjacent
child body) and a $2$-chain $E$ such that
\[
 \partial E=C-C',\ 
 \mass_1(C')\le h_E(v,D)\,\diam(\Bb),\ 
 \mass_2(E)\le C(v,D)\,\rho\,\mass_1(C),
\]
where $h_E(v,D)$ depends only on the uniform bound $C_\Gamma(v,D)$ from
Theorem~\ref{thm:bubble_tree_Ein}.
\end{lemma}

\begin{proof}
The argument is exactly as in~\cite[Sec.~4.2]{LW_HF1}.  Via $\Phi^{-1}$ we
view $C$ as a $1$-cycle in the annulus $A_{\rho,2\rho}(0)\subset\R^4/\Gamma$.
Radial projection onto a sphere $S_\lambda/\Gamma$ with
$\lambda\in(\rho,2\rho)$ sweeps out a $2$-chain of area $\lesssim
\rho\,\mass_1(C)$. The projected cycle is a $1$-cycle in $S^3/\Gamma$, whose
first homology is finite with generators represented by loops of length
$\le 2\pi\lambda$. Replacing the projected cycle by a representative of its
homology class supported in a bounded number of such generators yields $C'$
with $\mass_1(C')\le h_E(v,D)\,\lambda$.  Pushing forward by $\Phi$ and
controlling the distortion inside the annulus completes the proof.
\end{proof}

\subsection{Approximation by $\Gamma$ and improved combinatorial filling}

The next step is to replace a $1$-cycle lying in a body region by a combinatorial
cycle in the geodesic graph $\Gamma$.

\begin{lemma}\label{lem:cycle-to-graph}
Let $C$ be a $1$-cycle supported in a body region $\Bb$. Then there exists
a $1$-cycle $C'\subset \Gamma$ and a $2$-chain $E$ such that
\[
 \partial E=C-C',\ 
 \mass_1(C')\le C_\Gamma(v,D)\,\mass_1(C),\ 
 \mass_2(E)\le C_\Gamma(v,D)\,\mass_1(C).
\]
\end{lemma}

\begin{proof}
The proof is identical to~\cite[Lemma~4.5]{LW_HF1}. One partitions $C$ into
short arcs contained in the balls of the cover $\Oo(\Bb)$, with length
$\le R^i_j$. Inside each ball one joins the endpoints of an arc to the
corresponding center along edges of $\Gamma$. The difference between the
original arc and the resulting path in $\Gamma$ lies in a harmonic ball and
hence can be filled using Lemma~\ref{lem:fill-ball}. Summing over all arcs
gives the desired estimates.
\end{proof}

The idea of controlling the combinatorial filling via integer solutions of linear systems was introduced in~\cite{nabutovsky2006curvature}. We show that one can obtain an upper bound for the number of simplices needed to fill a simplicial cycle which depends linearly on the size of the cycle. Note that this is independent of the dimension of the simplicial complex.

Let $A\in\mathbb{Z}^{\ell\times n}$, $b\in\mathbb{Z}^\ell$, and assume
$\mathrm{rank}(A)=m\le \ell$. Set
\[
 M_A := \max_{1\le i\le \ell,\;1\le j\le n} |a_{ij}|,\qquad
 M_b := \max_{1\le i\le \ell} |b_i|.
\]
Let $Y$ be the maximum of the absolute values of all $m\times m$ minors of the
augmented matrix $(A\,|\,b)$:
\[
 Y := \max\bigl\{\,|\det S|:\ S\text{ is an $m\times m$ submatrix of }(A\,|\,b)\,\bigr\}.
\]

\begin{theorem}\label{thm:BFRT-Hadamard}
Assume that the system $Ax=b$ has an integer solution. Then there exists an
integer solution $x=(x_i)\in\mathbb{Z}^n$ such that
\begin{equation}\label{eq:BFRT-Hadamard-bound}
 \max_i |x_i|
 \;\le\; m^{m/2}\,\max\{\,M_A^m,\;M_A^{m-1}M_b\,\}
 \;=\; m^{m/2} M_A^{m-1}\max\{M_A,M_b\}.
\end{equation}
In particular, this upper bound depends only on $m$, $M_A$ and $M_b$
(and is independent of $n$).
\end{theorem}

\begin{proof}
By \cite[Thm.~1]{BFRT1989}, if $Ax=b$ has an integer solution then there exists
an integer solution $x$ with
\begin{equation}\label{eq:BFRT-Y}
 \max_i |x_i|\;\le\; Y,
\end{equation}
where $Y$ is the maximum of the absolute values of the $m\times m$ minors of
$(A\,|\,b)$ defined above.

We estimate $Y$ using Hadamard's inequality. Let $S$ be any $m\times m$
submatrix of $(A\,|\,b)$ and denote its columns by
$s_1,\dots,s_m\in\mathbb{R}^m$. Hadamard's inequality gives
\begin{equation}\label{eq:hadamard}
 |\det S|\;\le\;\prod_{k=1}^m \|s_k\|_2.
\end{equation}
There are two cases.

\smallskip\noindent
\emph{Case 1: $S$ is formed entirely by columns of $A$.}
Then every entry of $S$ has absolute value at most $M_A$, hence each column
has Euclidean norm
\(
 \|s_k\|_2 \le \sqrt{m}\,M_A.
\)
By~\eqref{eq:hadamard},
\[
 |\det S|
 \;\le\;(\sqrt{m}M_A)^m
 \;=\;m^{m/2}M_A^m.
\]

\smallskip\noindent
\emph{Case 2: $S$ contains the column $b$.}
Then one column of $S$ (the one coming from $b$) has norm at most
$\sqrt{m}\,M_b$, while each of the remaining $m-1$ columns has norm at most
$\sqrt{m}\,M_A$. Thus
\[
 |\det S|
 \;\le\;(\sqrt{m}M_b)\,(\sqrt{m}M_A)^{m-1}
 \;=\;m^{m/2}M_A^{m-1}M_b.
\]

Taking the maximum over all $S$ in these two cases yields
\[
 Y
 \;\le\; m^{m/2}\,\max\{M_A^m,\;M_A^{m-1}M_b\}.
\]
Combining this with~\eqref{eq:BFRT-Y} gives
\eqref{eq:BFRT-Hadamard-bound}.
\end{proof}

\begin{lemma}\label{lem:comb-improved}
Let $X$ be an $n$-dimensional simplicial complex with $n_0$ vertices.
Let
\[
 c_k = \sum_i a_k^i\,\sigma_i^k
\]
be a simplicial $k$-boundary in $X$ with integer coefficients.
Then there exists a simplicial $(k+1)$-chain
\[
 c_{k+1} = \sum_i a_{k+1}^i\,\sigma_i^{k+1}
\]
such that $\partial c_{k+1}=c_k$ and
\begin{equation}\label{eq:comb-max-coeff}
 \max_i |a_{k+1}^i|
 \;\le\; \binom{n_0}{k+1}^{\frac12\binom{n_0}{k+1}}\,
          \max_i |a_k^i|.
\end{equation}
In particular, since the number $n_{k+1}$ of $(k+1)$-simplices satisfies
$n_{k+1}\le \binom{n_0}{k+2}$, we also have
\begin{equation}\label{eq:comb-l1-bound}
 \sum_i |a_{k+1}^i|
 \;\le\; \binom{n_0}{k+2}\,\binom{n_0}{k+1}^{\frac12\binom{n_0}{k+1}}\,
          \max_i |a_k^i|.
\end{equation}
\end{lemma}

\begin{proof}
If $c_k=0$, the conclusion is trivial with $c_{k+1}=0$, so we assume
$c_k\neq 0$.  Let $C_j$ denote the group of $j$-chains of $X$ with integer
coefficients. We fix the usual basis of $C_j$ given by all oriented
$j$-simplices $\sigma^j_1,\dots,\sigma^j_{n_j}$, where $n_j$ is the number of
$j$-simplices of $X$.

Consider the boundary map
\[
 \partial_{k+1}\colon C_{k+1} \to C_k
\]
and write its matrix with respect to these bases as
\(
 A\in\mathbb{Z}^{n_k\times n_{k+1}}.
\)
By construction, each entry of $A$ lies in $\{0,\pm1\}$, so
\[
 M_A := \max_{i,j}|a_{ij}| = 1.
\]
Write $c_k$ in coordinates as
\(
 c_k = \sum_i a_k^i\sigma_i^k,
\)
and let $b\in\mathbb{Z}^{n_k}$ be the column vector with entries
$b_i = a_k^i$. Then the condition $\partial c_{k+1}=c_k$ is exactly the
linear system
\[
 A x = b,
\]
where $x\in\mathbb{Z}^{n_{k+1}}$ is the coordinate vector of $c_{k+1}$.
Since $c_k$ is a boundary, this system has at least one integer solution.
Set
\[
 M_b := \max_i |b_i| = \max_i |a_k^i|\ (\ge 1).
\]

Let $m:=\mathrm{rank}(A)\le n_k$. Applying
Theorem~\ref{thm:BFRT-Hadamard} with this $A$ and $b$, and using $M_A=1$,
we obtain an integer solution $x$ of $Ax=b$ such that
\begin{equation}\label{eq:max-ai+1}
 \max_i |x_i|
 \;\le\; m^{m/2}\,\max\{M_A,\;M_b\}
 \;=\; m^{m/2}\,\max_i |a_k^i|.
\end{equation}
Define $c_{k+1}:=\sum_i x_i \sigma_i^{k+1}$; then $\partial c_{k+1}=c_k$, and
the coefficients $a_{k+1}^i$ of $c_{k+1}$ are exactly the $x_i$.

Next we bound $m$ in terms of the number of vertices $n_0$. Any $k$-simplex
of $X$ is determined by its $(k+1)$ vertices, so the number of $k$-simplices
satisfies
\[
 n_k \le \binom{n_0}{k+1}.
\]
Hence
\[
 m = \mathrm{rank}(A) \le n_k \le \binom{n_0}{k+1}.
\]
Combining this with~\eqref{eq:max-ai+1} yields
\[
 \max_i |a_{k+1}^i|
 = \max_i |x_i|
 \;\le\; m^{m/2}\,\max_i |a_k^i|
 \;\le\; \binom{n_0}{k+1}^{\frac12\binom{n_0}{k+1}}\max_i |a_k^i|,
\]
which is exactly~\eqref{eq:comb-max-coeff}. Finally, since there are at most
$n_{k+1}\le \binom{n_0}{k+2}$ nonzero coefficients $a_{k+1}^i$, we also have
\[
 \sum_i |a_{k+1}^i|
 \;\le\; n_{k+1}\,\max_i |a_{k+1}^i|
 \;\le\; \binom{n_0}{k+2}\,\binom{n_0}{k+1}^{\frac12\binom{n_0}{k+1}}\,
          \max_i |a_k^i|,
\]
which is~\eqref{eq:comb-l1-bound}.
\end{proof}

In our application we will only use Lemma~\ref{lem:comb-improved} for $k=1$
and for the simplicial complex $X=\Nc$ (the nerve of $\Oo$).

\begin{lemma}[Filling geodesic triangles]\label{lem:triangles}
There exists a constant $K_9(v,D)>0$ such that any geodesic triangle formed
by three edges of $\Gamma$ bounds a $2$-chain $E$ with
\[
 \mass_2(E)\le K_9(v,D).
\]
\end{lemma}

\begin{proof}
As in~\cite[Lemma~4.6]{LW_HF1}. If all three vertices lie inside
a single body, the triangle lies in a harmonic ball $B_{r_h/2}$ and can be
filled using Lemma~\ref{lem:fill-ball}. If they lie inside a single neck, the
triangle lies in the union of at most two trapezoids contained in
$\bar T_{i,j,\ell}$, and can be filled by a cone in the annulus, with area
$\lesssim \rho^2$. Mixed configurations are reduced to these two cases.
All constants can be controlled in terms of $r_0(v,D)$ and $D$.
\end{proof}

\subsection{Decomposition and the Einstein filling inequality}

We now put the pieces together.  The first step is to decompose a general
$1$-cycle according to the bubble--tree structure.

\begin{lemma}[Decomposition along the tree]\label{lem:decompose}
Let $(M^4,g)\in\Ee_1(4,v,D)$ and let $C\in\Zz_1(M;\Z)$ be a singular
Lipschitz $1$-cycle. Then there exists a decomposition
\[
 C=\sum_{i=1}^N C_i,
\]
with $N\le N_E(v,D)\,k_E(v,D)$, such that each $C_i$ is supported in a single
body or a single neck and
\[
 \sum_{i=1}^N \mass_1(C_i)\;\le\; C_0(v,D)\,\mass_1(C).
\]
\end{lemma}

\begin{proof}
As in~\cite[Sec.~4.3]{LW_HF1}, one cuts $C$ whenever it crosses a boundary
component $\partial\Bb\cap\Nn$ and reassembles the resulting arcs into cycles
supported in individual bodies or necks. The number of such regions is bounded
by the bubble--tree complexity $N_E(v,D)\,k_E(v,D)$, and the total length
increases by at most a uniform multiplicative factor $C_0(v,D)$, coming from
bounded overlap of the decomposition and the uniform control on the diameter
of bodies and necks.
\end{proof}

\begin{corollary}[Reduction to bodies]\label{cor:reduce-to-bodies}
Let $C$ be as in Lemma~\ref{lem:decompose}. Then $C$ is homologous to a
$1$-cycle $C^{\mathrm{body}}$ supported in the union of body regions, i.e.
there exists a $2$-chain $E_0$ such that
\[
 \partial E_0=C-C^{\mathrm{body}},
\]
and
\[
 \mass_1(C^{\mathrm{body}})\le C_1(v,D)\,\mass_1(C),\ 
 \mass_2(E_0)\le C_2(v,D)\,\mass_1(C)+C_3(v,D).
\]
\end{corollary}

\begin{proof}
Apply Lemma~\ref{lem:neck-contraction} to each neck-supported component
$C_i$ in the decomposition of Lemma~\ref{lem:decompose}. This replaces each
$C_i$ by a cycle supported in the adjacent body at the expense of adding a
$2$-chain whose mass is controlled by $C(v,D)\,\rho\,\mass_1(C_i)$.
Summing over all necks and using the uniform bounds on the number and geometry
of necks yields the stated estimates.
\end{proof}

We are now ready to state and prove the main filling inequality in the Einstein
class.

\begin{theorem}[Einstein filling inequality]\label{thm:Ein_filling}
There exist $f_1^{\mathrm{Ein}}(v,D), f_2^{\mathrm{Ein}}(v,D)>0$
such that for any $(M^4,g)\in \Ee_1(4,v,D)$ and any $C\in \Zz_1(M;\Z)$ there
exists $E\in\Cc_2(M;\Z)$ with
\[
 \partial E=C,\ 
 \mass_2(E)\le f_1^{\mathrm{Ein}}(v,D)\,\mass_1(C)+f_2^{\mathrm{Ein}}(v,D).
\]
Equivalently, the first homological filling function satisfies
\[
 \HF_1(l)\le f_1^{\mathrm{Ein}}(v,D)\,l+f_2^{\mathrm{Ein}}(v,D)
 \ \text{for all }l\ge 0.
\]
\end{theorem}

\begin{proof}
By Corollary~\ref{cor:reduce-to-bodies}, $C$ is homologous to a cycle
$C^{\mathrm{body}}$ supported in the union of bodies, with
\[
 \partial E_0=C-C^{\mathrm{body}},\quad
 \mass_2(E_0)\le C_2(v,D)\,\mass_1(C)+C_3(v,D),
\]
and
\(
 \mass_1(C^{\mathrm{body}})\le C_1(v,D)\,\mass_1(C).
\)

Next, approximate $C^{\mathrm{body}}$ by a cycle in the geodesic graph
$\Gamma$ using Lemma~\ref{lem:cycle-to-graph}. We obtain a cycle
$\widehat C\subset\Gamma$ and a $2$-chain $E_1$ such that
\[
 \partial E_1=C^{\mathrm{body}}-\widehat C,\ 
 \mass_2(E_1)\le C_\Gamma(v,D)\,\mass_1(C^{\mathrm{body}})
 \le C_\Gamma(v,D)\,C_1(v,D)\,\mass_1(C).
\]

The remaining step is the chain-level passage from the graph $\Gamma$ to the
nerve $\Nc$. Instead of using the partition-of-unity map directly on
$\widehat C$, we invoke the graph-to-nerve straightening procedure from
\cite[Sec.~4]{LW_HF1}. Applied to the graph cycle $\widehat C$, it yields a
simplicial $1$-boundary
\[
 c_1=\sum_i a_1^i\,\sigma_i^1
\]
in $\Nc$ and a $2$-chain $E_{\mathrm{st}}$ in $M$ such that
\[
 \partial E_{\mathrm{st}}=\widehat C-\gamma(c_1),
\]
where $\gamma(c_1)$ denotes the edge cycle in $\Gamma$ corresponding to
$c_1$, and
\[
 \sum_i |a_1^i| \le C_L(v,D)\,\mass_1(C),
 \qquad
 \mass_2(E_{\mathrm{st}})\le C_L(v,D)\,\mass_1(C).
\]
In particular,
\[
 \max_i|a_1^i|\le C_L(v,D)\,\mass_1(C).
\]

Let $n_0$ be the number of vertices of $\Nc$, so $n_0\le \ti N_E(v,D)$.
Since $c_1$ is a simplicial boundary in $\Nc$, Lemma~\ref{lem:comb-improved}
with $k=1$ yields a simplicial $2$-chain
\(
 C_2=\sum a_2^i\sigma_i^2
\)
with $\partial C_2=c_1$ and
\[
 \sum_i |a_2^i|
 \;\le\; B_1(n_0)\,\max_i|a_1^i|
 \;\le\; B_1(\ti N_E(v,D))\,C_L(v,D)\,\mass_1(C),
\]
where $B_1(n_0)$ is the combinatorial coefficient from
\eqref{eq:comb-l1-bound}. Each $2$-simplex in $\Nc$ corresponds to a geodesic
triangle in $\Gamma$, which bounds a $2$-chain in $M$ of area at most
$K_9(v,D)$ by Lemma~\ref{lem:triangles}. Thus $C_2$ lifts to a $2$-chain
$E_2$ in $M$ with
\[
 \partial E_2=\gamma(c_1),\ 
 \mass_2(E_2)
 \le K_9(v,D)\,\sum_i |a_2^i|
 \le K_9(v,D)\,B_1(\ti N_E(v,D))\,C_L(v,D)\,\mass_1(C).
\]

Finally, set
\[
 E:=E_0+E_1+E_{\mathrm{st}}+E_2.
\]
Then $\partial E=C$, and combining the three estimates we obtain
\[
 \mass_2(E)
 \le \bigl(C_2+C_\Gamma C_1
           +C_L+K_9 B_1(\ti N_E)C_L\bigr)\mass_1(C)\;+\;C_3.
\]
We may therefore take
\[
 f_1^{\mathrm{Ein}}(v,D)
 := C_2+C_\Gamma C_1
    +C_L+K_9 \,B_1(\ti N_E)\,C_L,
\]
and
\[
 f_2^{\mathrm{Ein}}(v,D):=C_3(v,D).
\]
By construction, all constants involved depend only on the bubble--tree data
$r_0(v,D)$, $C_\Gamma(v,D)$, $\ti N_E(v,D)$, $N_E(v,D)$, $k_E(v,D)$ and on
the Sobolev/$L^2$-regularity inputs from Section~\ref{sec:Einstein_prelim};
in particular, they depend only on $(v,D)$.
\end{proof}

\section*{Acknowledgements}
The author Z. Zhu was supported in part by National Key R\&D Program of China 2024YFA1015300 and NSFC 1250010531.

\end{document}